\documentclass[a4paper]{article}

\usepackage{silence}
\usepackage[utf8]{inputenc}
\usepackage{amsthm}
\usepackage{amsmath}
\usepackage{amssymb}
\usepackage{caption}
\usepackage{cite}
\usepackage[labelformat=simple]{subcaption}

\usepackage{a4wide}
\usepackage{todonotes}
\usepackage{multicol}

\usepackage{tikz}
\usetikzlibrary{calc, fit, decorations.pathreplacing}
\usepackage{tkz-graph} 
\tikzset{cycle2/.style={very thick, densely dashed}}
\usetikzlibrary{shapes.geometric}

\usepackage{tkz-berge}

\usepackage{changes}

\usepackage{hyperref}
\usepackage{color}

\newtheorem{theorem}{Theorem}

\newtheorem{corollary}[theorem]{Corollary}

\newtheorem{lemma}[theorem]{Lemma}

\theoremstyle{definition}
\newtheorem{definition}[theorem]{Definition}

\newtheorem{example}{Example}
\theoremstyle{remark}   

\newtheorem{observation}{Observation}[theorem]

 \usepackage{enumerate}

\newcommand{\abs}[1]{\left\vert{#1}\right\vert} 

\newcommand{\sep}{\mid}

\usepackage{mathtools}

\DeclareMathOperator{\cc}{cc}
\DeclareMathOperator{\cy}{cy}
\DeclareMathOperator{\mfv}{pf}

\newcommand{\bull}{B}
\newcommand{\AFV}{A}

\usepackage[mathcal]{euscript}
\newcommand{\clause}{\mathcal{C}}
\newcommand{\instance}{F}
\newcommand{\literals}{U}
\newcommand{\assignment}{A}

\newcommand{\auxiliary}{v_{-1}}

\newcommand{\koodiswap}[3]{#1[#2 \leftarrow #3]}

\title{On the Vertices That Belong to All Minimum Identifying Codes}

\author{\textbf{Ville Junnila}, \textbf{Tero Laihonen} and \textbf{Havu Miikonen}\\
Department of Mathematics and Statistics\\
University of Turku, FI-20014 Turku, Finland\\
viljun@utu.fi, terolai@utu.fi and havu.e.miikonen@utu.fi}

\date{}

\begin{document}

\maketitle
\begin{abstract}
Identifying codes in graphs have been widely studied since their introduction by Karpovsky, Chakrabarty and Levitin in 1998. In this paper, we consider the vertices that are in every minimum identifying code in a graph. There are two types of such vertices: \emph{always-forced} vertices that belong to all identifying codes  (minimum or not) and \emph{min-forced} vertices that belong to all minimum identifying codes. A vertex is called \emph{proper-min-forced} if it is min-forced but not always-forced. We show an upper bound $2n/3$ for the number of such proper-min-forced vertices in a closed-twin-free graph of order $n$. Moreover, for integers $n$ divisible by three, we construct an infinite family of graphs in which there are $2n/3-1$ such vertices. In addition, we determine the maximum number of edges in a graph of even order such that the graph contains proper-min-forced vertices. 
We also show that the decision problem of determining whether a given vertex in a graph is proper-min-forced is co-NP-hard.

\smallskip

\noindent\textbf{ AMS Subj.\ Class.\ (2020)}: 05C69, 05C42, 68Q15
\end{abstract}

\section{Introduction}

Let $G = (V(G), E(G))$ be a simple, finite and undirected graph.
We often denote the order of a graph $\abs{V(G)}$ by $n$. 
A graph is \emph{nontrivial} if $n\ge 2$.  A subgraph $H$ of $G$ is  \emph{spanning}  if $V(H)=V(G)$.
An edge between vertices $u,v \in V(G)$ is denoted by $\{u,v\} = uv$.
We denote by $N(v)$ the \emph{open neighbourhood} of a vertex $v$ which is defined as $N(v) = \{ u\in V \sep vu\in E(G) \}$. The \emph{closed neighbourhood} of a vertex $v \in V$ is $N[v] = N(v) \cup \{v\}$. If we wish to emphasize the underlying graph, we denote $N_G(v)$ and $N_G[v].$
Distinct vertices $u$ and $v$ are called \emph{closed twins} if $N[u] = N[v]$, and \emph{open twins} if $N(u) = N(v)$. A graph is called \emph{closed-twin-free} if there are no closed twins in it.

A vertex  $u\in V(G)$ is called \emph{universal} if $N_G[u]=V(G)$. We define the \emph{degree} of a vertex $v$ as  $d_G(v)=|N(v)|$, and the \emph{maximum degree} $\Delta(G)=\max\{d_G(v)\mid v\in V(G)\}$. The  \emph{degree sequence} $(d_1,d_2,\dots, d_n)$ is an (ascending) list of degrees of all the vertices. 
A nonempty subset $S \subseteq V(G)$ is called a \emph{code} and its
elements  are called \emph{codewords}.
We use the following shorthand notation for replacing a vertex $u$ of $S$ by another vertex $v$ (which usually does not belong to $S$): $ S[u \leftarrow v]=(S \setminus \{u\}) \cup \{v\}$. 
The \emph{symmetric difference} of sets $A$ and $B$  is denoted  by $A\, \triangle \, B=(A \setminus B)\cup (B \setminus A)$. Since $(A \,\triangle \, B)\, \triangle \, B=A$, it follows that 
\begin{equation}\label{SymmErotCancelointi}
    A \, \triangle \, B = C \, \triangle \, B \Rightarrow  A = C.
\end{equation}
We denote the \emph{complete graph} of order $n$ by $K_n$ and the \emph{path} of order $n$ by $P_n$.
The \emph{complement} of a graph $G$ is denoted by  $\overline{G}$. For a code $S \subseteq V(G)$ and a vertex $v \in V(G)$, we defined the $I$\emph{-set} $I_G(S;v)$ of $v$ as the set of codewords in the closed neighbourhood of $v$, that is, 
\[
I_G(S;v)  = S \cap N_G[v] \textbf{}.
\]
We may omit the graph or the code in the notation if they are clear from the context, that is, $I_G(S;v) = I_G(v) = I(S;v) = I(v)$.

The following concept of \emph{identifying codes} was introduced in~\cite{Karpovsky} in 1998. The topic has been studied widely in the literature, see the numerous articles in the list of~\cite{Lobstein}. For some recent developments on the subject, consult, for example, \cite{Dipayan}, \cite{Aline}, \cite{SukSeo} and\cite{Sampaio}.

\begin{definition}
    A code $S \subseteq V(G)$ is an \emph{identifying code}, or an \emph{ID code}, in a graph $G$ if for all distinct vertices $u,v \in V(G)$ their $I$-sets are nonempty and
    \(I(u) \neq I(v)\).
\end{definition}

The original application for identifying codes
was fault diagnosis in multiprocessor systems~\cite{Karpovsky}. Later, these codes were extended to applications for
environmental monitoring \cite{Berger-Wolf}, and joint monitoring and routing
in wireless sensor networks \cite{Laifenfeld}. Let us look more closely at the idea behind using identifying codes  to locate objects in a sensor network.
Sensor networks are systems consisting of sensors and links between them and, therefore, can be easily modeled as a graph $G$. Say we have a set of sensors $S\subseteq V(G)$. A sensor is placed at some node of a network, that is, at a vertex $c\in S$, and  it monitors its surroundings, that is, the closed neighbourhood $N_G[c]$, reporting on the possible existence of the sought object. Based on these reports, the central unit tries to deduce the location of the object. This can be done if and only if $S$ is an identifying code since each $I(v)$ is nonempty and pairwise distinct for all $v\in V(G)$. A natural goal is to minimize the number of sensors in a network, and, therefore, we are interested in finding identifying codes with \emph{minimum} cardinality. More information on location in sensor networks can be found, for instance, in~\cite{Trachtenberg,Laifenfeld}.

Clearly, if $S\subseteq V(G)$ is an identifying code, then a superset $S'\subseteq V(G)$ of $S$, that is, $S\subseteq S'$, is also an identifying code. In particular, this holds for $S'=V(G)$, and in that case, $I(v)=N[v]$ for all $v \in V(G)$. Therefore, in order to have an ID code in a graph, it must be closed-twin-free.
Closed-twin-freeness is also sufficient: the set $V(G)$ is an ID code, if $G$ is closed-twin-free.  The cardinality of the minimum identifying codes in a closed-twin-free graph $G$ is denoted by $\gamma^{ID}(G)$ and is called the \emph{identifying number} of $G$.

By~\cite{GRAVIER2007432}, it is known that if a finite $G$ is closed-twin-free and $G \neq \overline{K_n}$, there always exists a vertex $v$ such that $V(G)\setminus \{v\}$ is an ID code, and thus, $\gamma^{ID}(G)\le n-1$.

However, there can be vertices  that cannot be omitted from \emph{any} identifying code in a finite graph, as defined next.

\begin{definition} Let $G$ be a closed-twin-free graph.
A vertex $v \in V(G)$ is \emph{always-forced} if it is in every identifying code in $G$. In other words, $v$ belongs to the intersection $\bigcap_{S} S$, where $S$ goes through all identifying codes in $G$. 
\end{definition}

It is easy to see that a vertex $v$ is always-forced if and only if $V\setminus \{v\}$ is not an identifying code in $G.$ Consequently, it is algorithmically easy to check if a vertex is always-forced.
Clearly, isolated vertices in a graph are always-forced.
If  $N_G[u] \, \triangle \, N_G[w] = \{v\}$, then the vertex $v$ must be in every identifying code due to the pair $u,w$. In the literature (see e.g.,\cite{FlorentForced,FlorentGuillem}),  such vertices $v$ are called \emph{forced} or $u,w$-\emph{forced}. Obviously, an isolated vertex is not forced by any pair (although it is always-forced). The definition of always-forced coincides with this concept of forced (by a pair) in closed-twin-free graphs with no isolated vertices. This is seen by the next easy lemma.

\begin{lemma}\label{Connection}
     Let $G$ be a connected, nontrivial and closed-twin-free graph.
    A vertex $v \in V(G)$ is always-forced if and only if there exists a pair of vertices $u, w \in V(G)$ such that $N_G[u] \, \triangle \, N_G[w] = \{v\}$.
\end{lemma}
\begin{proof}

    ($\Rightarrow$) Assume first that $v$ is always-forced. Hence, $S=V(G)\setminus\{v\}$ is not an ID code. Then necessarily $I(S;x)=\emptyset$ or $I(S;u)=I(S;w)$ for some $x,u,w\in V(G)$, where $u\neq w$. However, since $G$ is connected and nontrivial, we have $I(S;x)\neq \emptyset$ for all $x\in V(G)$. Hence, we must have $I(S;u)=I(S;w)$ and this implies that $N_G[u] \, \triangle \, N_G[w] = \{v\}$ as $G$ is closed-twin-free.

    ($\Leftarrow$) This direction is trivial.
 \end{proof}

Recall that $\gamma^{ID}(G)\le n-1$ and hence, the number of always-forced vertices is also upper bounded by $n-1$. Consider a graph $G$ formed from a complete graph $K_{2k+1}$ by removing edges of a maximal matching. Denoting the universal vertex of $G$, that is, the vertex adjacent to all other vertices of $G$, by $u$, we may easily observe by studying the symmetric differences of the closed neighbourhoods of $u$ and $v \ (\neq u)$ that all the vertices belonging to $V(G) \setminus \{u\}$ are always-forced; the result also follows by Lemma~\ref{lemma:XY}(ii), which is later presented in Section~\ref{DenseSec}. Thus, $G$ attains the previous upper bound $n-1$ on the number of always-forced vertices. 

As discussed above, we are mostly interested in the smallest possible identifying codes, the minimum ID codes, and the vertices in such a code. With that in mind we give the following definition.

\begin{definition} Let $G$ be a closed-twin-free graph.
A vertex $v \in V(G)$ is \emph{minimum-forced} or \emph{min-forced} if it is in every minimum identifying code in $G$. In other words, $v$ belongs to the intersection $\bigcap_{S} S$, where $S$ goes through all minimum identifying codes in $G$. 
\end{definition}

Questions related to these vertices are rather natural and important when we try to build minimum identifying codes; indeed, such vertices must always be chosen in a minimum ID code. 
Similar questions have also been studied concerning other special subsets of graphs; for example, (total) dominating sets~\cite{Bouquet2021,Cockayne2003,Mynhardt1999}, locating-dominating sets~\cite{Blidia2009minforcedtrees,Junnila2026LD}, independent (or stable) sets \cite{max-stable-sets1999,Hammer1982}, matchings \cite{unique-perfect-matching2018} and resolving sets~\cite{Bagheri2016,Buczkowski2003,Hakanen2022metric, Hakanen2026edge}. Furthermore, the graphs for which each vertex belongs to some optimal (smallest or largest depending on the problem) special subset are called \emph{excellent} and have been studied, for example, in the case of dominating sets~\cite{Fricke2002,Samodivkin2021} and independent sets~\cite{Dettlaff2023}.

Clearly, always-forced vertices are in every ID code, so trivially, they are also in every minimum ID code. 
As explained earlier, determining whether a vertex is always-forced is easy.
Hence, the main focus of the paper is on the vertices that are min-forced but not always-forced. 
In the following definition, we call such vertices proper-minimum-forced or proper-min-forced.

\begin{definition} Let $G$ be a closed-twin-free graph.
A \emph{proper-min(imum)-forced vertex} is a min-forced vertex $v \in V(G)$ that is not always-forced.
\end{definition}

Recall that the maximum number of always-forced vertices in a graph of order $n$ is $n-1$ and that the upper bound can also be attained. In Section~\ref{2n3bound}, we show that the situation is very different for proper-min-forced vertices. Indeed, we show in Theorem~\ref{theo:bound} that there can only be at most $2n/3$ proper-min-forced vertices. In Section~\ref{DenseSec},  we give an upper bound on the  maximum number of edges in a graph of even order such that the graph contains proper-min-forced vertices. In addition, we show that the bound is tight. 
The question of a vertex being always-forced is easy to check algorithmically, but the situation is more complicated for the proper-min-forced vertices. Indeed, it is shown in Section~\ref{CompSec} that deciding whether a vertex is proper-min-forced is co-NP-hard.

\section{On the maximum number of proper-min-forced vertices}\label{2n3bound}

 Next we define a useful concept of a graph which has coloured edges. It helps us give an upper bound $2n/3$ on the number of proper-min-forced vertices (see Theorem~\ref{theo:bound}). In Example~\ref{shovelgraph}, we build an infinite family of graphs containing $2n/3-1$ proper-min-forced vertices when $n$ is divisible by three.

\begin{definition} \label{IDColourGraph}
    Let $G$ be a connected, nontrivial and closed-twin-free graph as well as let $S \subseteq V(G)$ be an identifying code in $G$.
    We define the \emph{colour graph} $G_S$ as follows: $V(G_S) = V(G) \cup \{\auxiliary\}$ and 
    \begin{align*}
        E(G_S) =& \left\{xy \sep x, y \in V(G),\ x \neq y \text{ and } I_G(S; x) \, \triangle \, I_G(S; y) = \{u\} \text{, where } u \in S \right\} \\
        &\cup \left\{x\auxiliary \sep x \in V(G) \text{ and } I_G(S; x)  = \{u\} \text{, where } u \in S \right\}.
    \end{align*}
    We call vertices in $V(G)$ \emph{true vertices} and the vertex $\auxiliary$ an \emph{auxiliary vertex}.
    Associating a colour with each codeword $u \in S$, we assign the colour $u$ to the edge $xy \in E(G_S)$ if $I_G(S; x) \, \triangle \, I_G(S; y) = \{u\}$, and similarly,  to $x \auxiliary$ if $I_G(S; x)  = \{u\}$.
\end{definition}
Notice that the colour assigned to an edge this way in $G_S$ is unique and every edge gets a colour. An example of a colour graph can be found in Figure~\ref{fig:example}.
A similar concept of a colour graph has been studied in \cite{hernando2018locating, Junnila2026LD} for locating-dominating codes and in \cite{Hakanen2022metric, Hakanen2025tight-bound} for resolving sets. However, because the definitions of identifying codes, locating-dominating codes and resolving sets differ, there are significant differences between the colour graph of Definition~\ref{IDColourGraph} and the colour graphs related to locating-dominating codes and resolving sets. 
In particular, the concept of always-forced vertices is unique for identifying codes. Moreover, the upper bound obtained for identifying codes ($\approx 2n/3$) is significantly larger than the upper bounds ($2n/5$) in the cases of locating-dominating codes and resolving sets. Next, we provide some properties of the graph $G_S$ defined above, which share some similarities with locating-dominating codes \cite{hernando2018locating, Junnila2026LD} and resolving sets \cite{Hakanen2022metric, Hakanen2025tight-bound}.

\begin{lemma}
\label{lemma:G_S-properties}
Let $G$ be a connected, nontrivial and closed-twin-free graph as well as let $S$ be an identifying code in $G$.
    \begin{enumerate}[(i)]
    
        \item Let there be an edge $xy \in E(G_S)$ with colour $u$.
        If $x$ and $y$ are true vertices, then $u \in N_G[x] \, \triangle \, N_G[y]$.
        If $y = \auxiliary$ (resp. $x = \auxiliary$), then $u \in N_G[x]$ (resp. $u \in N_G[y]$).

        \item Any incident edges $xy$ and $yz$ of $G_S$ have distinct colours.
    
        \item If $S$ is a \emph{minimal} identifying code, then there is at least one edge in $G_S$ with  colour $u$ for each codeword $u \in S$.

        \item For each always-forced vertex $v$, there is an edge $uw \in E(G_S)$ with  colour $v$ such that $N_G[u] \, \triangle \, N_G[w] = \{v\}$ (regardless of whether $S$ is minimal or not).
        
        \item Every cycle (if any) in $G_S$ has an even number of edges with  colour $u$ for each $u$ appearing on the cycle. Moreover, the graph $G_S$ is bipartite.
        
        \item Let $W = w_1 w_2 \cdots w_k$ be a walk with no repeated edges in $G_S$.
        If each colour appearing on the edges of the walk $W$ occurs an even number of times, then $W$ is a closed walk, that is, $w_1 = w_k$.

    \end{enumerate} 
\end{lemma}
\begin{proof}
(i) If $x$ and $y$ are true vertices, we obtain by the definition of the colour graph, that $I_G(S; x) \, \triangle \, I_G(S; y) = \{u\}$. The claim $u \in N_G[x] \, \triangle \, N_G[y]$ follows from this directly. If $y=\auxiliary$ (resp. $x=\auxiliary$), then $I_G(S;x)=\{u\}$ which implies that $u\in N_G[x]$ (resp. $u\in N_G[y])$.

(ii) Suppose to the contrary that there are edges $xy$ and $yz$ with  colour $u$.
        Let us first assume that $x$, $y$ and $z$ are all true vertices. Now, by definition,  $I_G(S; x) \, \triangle \, I_G(S; y) = \{u\}$ and  $I_G(S; z) \, \triangle \, I_G(S; y) = \{u\}$. This implies by \eqref{SymmErotCancelointi} that $I_G(S; x) = I_G(S; z)$, contradicting the fact that $S$ is an identifying code.

        If $y = \auxiliary$, then the edges $xy$ and $yz$ with colour $u$ imply that $I_G(S;x) = \{u\} = I_G(S;z)$, which contradicts the identifying property of $S$.

        Finally, if $x = \auxiliary$ (or $z = \auxiliary$), the $u$-coloured edge $xy$ implies that $I_G(S;y) = \{u\}$. Now since $I_G(S; y) \, \triangle \, I_G(S; z) = \{u\}$, it must be that $I_G(S; z) = \emptyset$, contradicting the dominating property of $S$.
        
(iii) Suppose that there is a codeword $u$ such that there are no edges in $E(G_S)$ with  colour associated with $u$.
        Let us now look at the code $S \setminus \{u\}$.
        The $I$-sets of all vertices with respect to $S \setminus \{u\}$ are nonempty, because if there was a vertex $w$ the $I$-set $I(S;w) = \{u\}$ (and consequently, $I(S\setminus \{u\};w) = \emptyset$), there would be a $u$-coloured edge $w\auxiliary$ in $E(G_S)$.

Since $I(S;x) \, \triangle \, I(S;y) \neq \{u\}$ (and $I(S;x) \, \triangle \, I(S;y) \neq \emptyset$) for all pairs of distinct vertices $x, y \in V(G)$, there exists a vertex $a_{xy} \neq u$ such that $a_{xy}\in I(S,x) \, \triangle \, I(S;y)$. It follows that $I(S \setminus \{u\};x)   \neq I(S \setminus \{u\};y)$ for all distinct $x$ and $y$. Hence, $S\setminus \{u\}$ is an ID code, which is a contradiction to the minimality of $S$.

(iv)   By Lemma~\ref{Connection}, for each always-forced vertex $v$ in a connected and nontrivial graph, there is a pair $u$ and $w$ with $N_G[u] \, \triangle \, N_G[w]=\{v\}$. It follows that for each ID code $S$, we have $I(S;u)\, \triangle \, I(S;w)=\{v\}$ and, thus, by the definition of the colour graph, $uw$ has the colour $v$.

(v) Let $C =  w_1 w_2 \cdots w_k w_1$ be a cycle in $G_S$. 
        Each edge $w_{i} w_{i+1}$ in $C$ corresponds to the removal (or addition) of a codeword $u$ from (or to) $I(S;w_i)$ to get $I(S;w_{i+1})$. 
        Each such removal (addition) has to be undone by adding (removing) the vertex $u$ somewhere along the cycle to get back to $I(S;w_i)$.
        Thus, each colour $u$ appears on the cycle $C$ an even number of times. 
        Due to the well-known result by K\H{o}nig, the graph $G_S$ with no odd cycles is bipartite.
        
        (vi) Each appearance of a colour $u$ along the walk $W =  w_1 w_2 \cdots w_k$ adds or removes the codeword $u$ from the $I$-set of a vertex on the walk.
        Starting from $I(S;w_1)$, the additions and removals of codewords cancel each other out and we get $I(S;w_k) = I(S;w_1)$.
        If $I(S;w_1) = I(S;w_k)$, then necessarily $w_1 = w_k$ as $S$ is an ID code.
        Thus, $W$ is a closed walk.
\end{proof}

 Let us consider the graph $G$ illustrated in  Figure~\ref{fig:example}(a) and its minimum ID code $S=\{1,2,3,6,7\}$. The vertices 6 and 7 are always-forced (since $N[4] \, \triangle \, N[8] = \{6\}$ and $N[1] \, \triangle \, N[5] = \{7\}$). By Lemma~\ref{lemma:G_S-properties}(iv) there must be at least one edge with colour 6 and 7 in the colour graph. Indeed, in the colour graph of Figure~\ref{fig:example}(b), there is  one 6-coloured edge (between vertices 4 and 8) and
the vertex 7 has two edges with its colour (between 1 and 5 and also between 4 and the auxiliary vertex $\auxiliary$). Figure~\ref{fig:example}(b) also illustrates the facts that all incident edges have different colours and each colour appearing in a cycle occurs an even number of times.

\begin{figure}[h] 
     \centering
     \begin{subfigure}[t]{0.45\textwidth}
        \centering
        \begin{tikzpicture}[yscale=0.6, xscale=0.6, rotate=0]

    \renewcommand*{\EdgeLineWidth}{ 0.5pt}
    \SetVertexMath
    \SetVertexLabelOut

    {
    \SetVertexNoLabel
    \tikzstyle{VertexStyle}=[minimum size=1pt,inner sep=0pt]

    \Vertex[x=6, y=2]{1}
    \Vertex[x=3.5, y=4]{2}
    \Vertex[x=0.3, y=2]{3}
    \Vertex[x=3.5, y=0]{4}
    \Vertex[x=8, y=1]{5}
    \Vertex[x=2, y=3]{6}
    \Vertex[x=4, y=2]{7}
    \Vertex[x=2, y=1]{8}

    }

    \tikzstyle{EdgeStyle}=[bend right=0]
    \Edge(1)(5)
    \Edge(1)(7)
    \Edge(2)(6)
    \Edge(2)(7)
    \Edge(3)(6)
    \Edge(4)(8)
    \Edge(4)(7)    
    \Edge(6)(7)
    \Edge(6)(8)
    \Edge(7)(8)
    \tikzstyle{EdgeStyle}=[bend right=40, looseness=1.1]

    \tikzstyle{VertexStyle}=[circle,draw, fill=red!30]
    \SO[unit=0, Lpos=0](2){2}
    \SO[unit=0, Lpos=90](3){3}

    \tikzstyle{VertexStyle}=[circle,draw, fill=red!80]
    \SO[unit=0, Lpos=90](1){1}
    
    \tikzstyle{VertexStyle}=[circle,draw, fill=black]
    \SO[unit=0, Lpos=45](7){7}
    \SO[unit=0, Lpos=135](6){6}

    \tikzstyle{VertexStyle}=[circle,draw, fill=white]
    \SO[unit=0, Lpos=180](8){8}
    \SO[unit=0](4){4}
    \SO[unit=0, Lpos=90](5){5}

\end{tikzpicture}
		\caption{}
     \end{subfigure}
     \hfill
     \begin{subfigure}[t]{0.45\textwidth}
        \centering
        
        \begin{tikzpicture}[yscale=0.6, xscale=0.6, rotate=0]

    \renewcommand*{\EdgeLineWidth}{ 0.5pt}
    \SetVertexMath
    \SetVertexLabelOut

    {
    \SetVertexNoLabel
    \tikzstyle{VertexStyle}=[minimum size=1pt,inner sep=0pt]

    \Vertex[x=6, y=2]{1}
    \Vertex[x=3.5, y=4]{2}
    \Vertex[x=0.3, y=2]{3}
    \Vertex[x=3.5, y=0]{4}
    \Vertex[x=8, y=1]{5}
    \Vertex[x=2, y=3]{6}
    \Vertex[x=4, y=2]{7}
    \Vertex[x=2, y=1]{8}

    }

    \begin{scope}
        \tikzstyle{VertexStyle}=[circle,draw, fill=white]
        \SetVertexLabelOut
        \Vertex[x=6, y=0, Lpos=90]{\auxiliary}
    \end{scope}

    \tikzstyle{LabelStyle}=[fill=white, circle,minimum size=12pt,inner sep=1pt,scale=.8]

    \tikzstyle{EdgeStyle}=[black]
    \Edge[label=$1$](1)(4)
    \Edge[label=$1$](5)(\auxiliary)
    \Edge[label=$1$](2)(7)
    \Edge[label=$2$](2)(8)
    \Edge[label=$3$](2)(6)
    \Edge[label=$6$](8)(4)
    \Edge[label=$7$](\auxiliary)(4)
    \Edge[label=$7$](1)(5)

    \tikzstyle{VertexStyle}=[circle,draw, fill=red!30]
    \SO[unit=0, Lpos=0](2){2}
    \SO[unit=0, Lpos=90](3){3}

    \tikzstyle{VertexStyle}=[circle,draw, fill=red!80]
    \SO[unit=0, Lpos=90](1){1}
    
    \tikzstyle{VertexStyle}=[circle,draw, fill=black]
    \SO[unit=0, Lpos=45](7){7}
    \SO[unit=0, Lpos=135](6){6}

    \tikzstyle{VertexStyle}=[circle,draw, fill=white]
    \SO[unit=0, Lpos=180](8){8}
    \SO[unit=0, Lpos=90](4){4}
    \SO[unit=0, Lpos=90](5){5}

\end{tikzpicture}
		\caption{}
     \end{subfigure}

    \caption{(a) A graph $G$ with a minimum ID code $S=\{1,2,3,6,7\}$. (b) The corresponding colour graph $G_S$. The edges are coloured by the labels of the codewords in $S$. Notice that the colour graph does not have to be connected (for example, the vertex $3$ is isolated) although $G$ is connected.}
    \label{fig:example}
\end{figure}
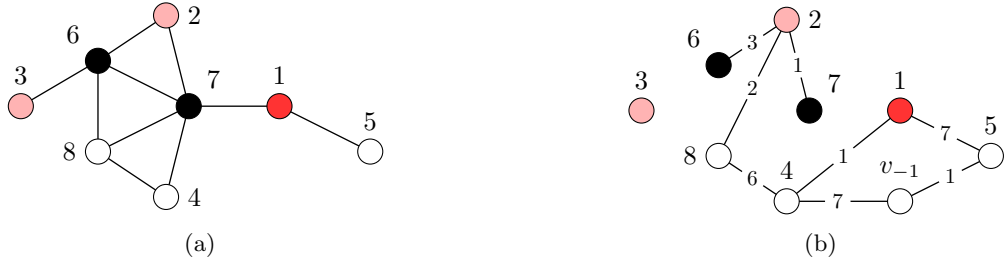

In the next lemma, we give a further property of colours in $G_S$.

\begin{lemma}
\label{lemma:atleast2edges}
Let $G$ be a connected, nontrivial  and closed-twin-free graph and assume that $S$ is an
identifying code in $G$.
Let $v \in S$ be a vertex that is not always-forced such that there is exactly one edge of colour $v$ in the colour graph $G_S$.
Then there exists a vertex $u \notin S$ such that the set $\koodiswap{S}{v}{u}$ is an identifying code in $G$.

\end{lemma}

\begin{proof}
Let $xy \in E(G_S)$ be the only edge with  colour $v$.
First, we assume that both $x$ and $y$ are true vertices.
By Lemma~\ref{lemma:G_S-properties}(i), $v \in N_G[x]\, \triangle \, N_G[y]$ (it is possible that $v =x$ or $v=y$).
Since we assumed that $v$ is not always-forced, it cannot be that $\{v\} = N_G[x]\, \triangle \, N_G[y]$, and therefore, there exists another vertex $u \in V(G)$ such that $u \in N_G[x]\, \triangle \, N_G[y]$.
Furthermore, we know that $u \notin S$, since $v$ is the only codeword in $S$ that distinguishes $x$ and $y$.
The set $\koodiswap{S}{v}{u}$ is an ID code in $G$ since now the vertex $u$ distinguishes $x$ and $y$, and no other pair of vertices relied solely on $v$ for identification, and all vertices are still dominated under $\koodiswap{S}{v}{u}$ since $x$ and $y$ were true vertices.

Then, let us assume that $y = \auxiliary$ (the case $x = \auxiliary$ goes similarly).
By Lemma~\ref{lemma:G_S-properties}(i), $v \in N_G[x]$.
The edge $x \auxiliary$ with  colour $v$ means that $v$ is the only vertex that dominates $x$ (it may be that $v=x$).
Since $G$ is connected and nontrivial, we cannot have $\{v\} = N_G[x]$.
It follows that there exists another vertex $u \in V(G)$ such that $u \in N_G[x]$.
The vertex $u$ is not in $S$, because $x$ is dominated only by $v$. 
The set $\koodiswap{S}{v}{u}$ is now an ID code in $G$.
 \end{proof}

In the following result, we consider the number of edges in $G_S$ which have a colour related to a proper-min-forced vertex. 
\begin{corollary}
\label{cor:atleast2edges-pmf}
    Let $G$ be a connected, nontrivial and closed-twin-free graph and let $S$ be a minimum identifying code in $G$.
    If $v \in S$ is proper-min-forced, then there are at least 2 edges with  colour $v$ in the graph $G_S$. 
\end{corollary}
\begin{proof}
    By Lemma~\ref{lemma:G_S-properties}(iii), there must be at least one edge with colour $v$ in $G_S$.
    By definition, proper-min-forced vertices are not always-forced.
    The set $\koodiswap{S}{v}{u}$ cannot be an ID code in $G$ for any choice of $u$ since $v$ is proper-min-forced and must belong to every minimum ID code.
    Therefore, by Lemma~\ref{lemma:atleast2edges}, there are at least 2 edges with colour $v$ in the graph $G_S$ for each proper-min-forced $v\in S$.
     \end{proof}

 In the code and colour graph of Figure~\ref{fig:example}, the vertex $1$ is proper-min-forced. It is guaranteed by Corollary~\ref{cor:atleast2edges-pmf} to have at least two $1$-coloured edges in $G_S$ and, in this example, it has three. However, there are graphs and codes with exactly two edges corresponding to a colour of a proper-min-forced vertex. Indeed, in the graph of Example~\ref{shovelgraph}, it is easy to verify that  the vertex $x_1$ has only two edges of its colour.

In what follows,  we consider certain subgraphs of a colour graph for which we need the concept of a   \emph{cactus graph}. A cactus graph is defined as a graph in which no two cycles share an edge. For completeness, we give the proof of the following known (see, \textit{e.g.}, \cite{hernando2018locating}) result.

\begin{lemma}
    \label{lemma:cactus_bound2}
    Let $H$ be a graph such that all its connected components are cacti. Then $\abs{V(H)} = \abs{E(H)} - \cy(H) + \cc(H)$, where $\cy(H)$ is the number of cycles in $H$ and $\cc(H)$ is the number of connected components in $H$.
\end{lemma}
\begin{proof}
By the Euler's formula for planar graphs, we get $\abs{V(H)} - \abs{E(H)}+ F - \cc(H) = 1$ where $F$ is the number of faces. 
Cactus graphs are planar, and since cycles are edge-disjoint, the number of faces in $H$ is $\cy(H) + 1$. We get
\[    \abs{V(H)} - \abs{E(H)}+ F - \cc(H) = 1 \Rightarrow \abs{V(H)}  =  \abs{E(H)}- \cy(H) + \cc(H).\]
\end{proof}

We also need an auxiliary lemma concerning cycles.
\begin{lemma}[Lemma~3 in \cite{Hakanen2025tight-bound}]
\label{lemma:cactus-theta}
Let $G$ be a graph and let $C_1$ and $C_2$ be (simple) cycles in $G$.
If $E(C_1)$ and $ E(C_2)$ have at least one edge in common, then there exists a third cycle $C_3$ in $G$ such that
the subgraph induced by the vertices of the intersection $V (C_1) \cap V (C_3)$ is a nontrivial path.

\end{lemma}

Next we will look at specific spanning subgraphs $H$ of the colour graph $G_S$ where we choose at most $2$ of the available edges with  the same colour.

\begin{lemma}
    \label{lemma:2_edges_cacti}
    Let $G$ be a connected, nontrivial and closed-twin-free graph and let $S$ be a minimum identifying code in $G$ of order $n$.
    \begin{enumerate}[(a)]
        \item     Choosing $0$ or $2$ edges of each colour of $G_S$  results in a subgraph whose connected components are bipartite cacti.
        \item    Choosing $0$, $1$ or $2$ edges of each colour of $G_S$ results in a subgraph whose connected components are bipartite cacti.
        \item    Choosing edges from $E(G_S)$ in such a way that $k$ colours appear once and $p$ colours appear twice  results in a spanning subgraph $H$ for which $n+1 \geq \frac{3}{2}p + k + \cc(H) \geq \frac{3}{2}p + k + 1$.
    \end{enumerate}    
\end{lemma}
\begin{proof}
First, notice that any subgraph of a bipartite graph is bipartite, and by Lemma~\ref{lemma:G_S-properties}(v), the graph $G_S$ is bipartite.
    \begin{enumerate}[(a)]
        \item 
        
        Suppose to the contrary that the subgraph $H_a$ is not composed of cacti, and thus the property ``no two cycles share an edge'' \emph{does not} hold.
        Then there are two cycles $C_1$ and $C_2$ such that $E(C_1)  \cap E(C_2)$ contains an edge.
        By Lemma~\ref{lemma:cactus-theta}, we may assume that the intersection of their edges form a nontrivial path $P$.
        We will name the other parts of the cycles with $C_1 \setminus C_2 = P^1$ and $C_2 \setminus C_1 = P^2$.
        Now, by Lemma~\ref{lemma:G_S-properties}(v), any colour $u$ that appears on the shared path $P$ must appear somewhere else on the cycles $C_1$ and $C_2$.
        If $u$ appears on $P^1$, then it cannot appear on $P^2$ by the choice of edges (in $H_a$ there were at most two of each chosen colour), making $C_2$ a cycle on which $u$ appears only once.
        By the same argument, $u$ cannot appear on $P^2$.
        Therefore, $u$ must appear twice on the shared path $P$.
        This holds for all colours that appear on $P$, making it a walk where each colour appears an even number of times, and by Lemma~\ref{lemma:G_S-properties}(vi), it is a closed walk giving a contradiction. Therefore,
cycles in $H_a$ do not have edges in common, that is, the connected components of $H_a$ are cacti.

        \item By the previous case (a), a (spanning) subgraph $H_a$ of $G_S$ with 0 or 2 edges of each colour is composed of cacti. In what follows, we add a new colour one at the time and the same new colour only once. Next we show that the property of connected components being cacti remains. Let us add these single colours recursively and
         assume that we have a spanning graph $H_b$ of $G_S$ composed of cacti with 0, 1 or 2 edges of each colour.
        Let us now create a graph $H'$ with $V(H') = V(G_S)$ and $E(H') = E(H_b) \cup \{xy\}$, where the $xy \in E(G_S)$ is of a colour that does not yet appear in $H_b$, say $w \in S$.
        The edge $xy$ cannot become part of a cycle in $H'$, since on such a cycle there were no other edges with  colour $w$, and by Lemma~\ref{lemma:G_S-properties}(v), each colour on a cycle appears an even number of times. 
        Since adding $xy$ does not create new cycles in $H'$, the property of $H_b$ concerning cycles is inherited by $H'$. Clearly, the subgraph $H'$ is also bipartite.
        
        \item There are $2p + k$ edges in $H$ where $p$ is the number of colours appearing twice and $k$ is the number of colours appearing exactly once.
        Based on the previous case~(b), $H$ is a bipartite graph, and hence each cycle of $H$ has at least $4$ vertices. Therefore, by Lemma~\ref{lemma:G_S-properties}(v), there can be at most $p/2$ cycles in $H$ (since none of the $k$ edges occurring exactly once can take part in a cycle). Combining these observations with Lemma~\ref{lemma:cactus_bound2}, we get $n+1= \abs{V(H)} = \abs{E(H)} - \cy(H) + \cc(H) \geq 2p + k - \frac{p}{2} + \cc(H) = \frac{3}{2}p + k + \cc(H) \geq \frac{3}{2}p + k + 1$.
    \end{enumerate}    
     \end{proof}

Finally, we are ready to prove our main theorem of this section.

\begin{theorem}\label{theo:bound}
    If $G$ is a connected, nontrivial and closed-twin-free graph of order $n$ with $\mfv(G)$ proper-min-forced vertices, then $\mfv(G)/2 + \gamma^{ID}(G) \leq n$ and, in particular, 
    \[
    \mfv(G) \leq \frac{2}{3}n.
    \]
\end{theorem}
\begin{proof}
Let $S \subseteq V(G)$ be a minimum identifying code in $G$.
By Corollary~\ref{cor:atleast2edges-pmf}, there are at least 2 edges with  colour $v$ in $G_S$ for each proper-min-forced vertex $v \in S$.
With that in mind, we select a specific subgraph $H$ of $G_S$ with $V(H) = V(G_S)$. 
Let $E(H)$ be a set of edges consisting of exactly $2$ edges of colour $v$ for each proper-min-forced $v$ and $1$ edge of colour corresponding to each other codeword (which is possible due to Lemma~\ref{lemma:G_S-properties}(iii)).
The graph $H$ is such that Lemma~\ref{lemma:2_edges_cacti}(c) applies with $p = \mfv(G)$ and $k = \gamma^{ID}(G)-\mfv(G)$. 
Therefore, we have 
\[
\frac{3}{2} \mfv(G) + (\gamma^{ID}(G)-\mfv(G)) +1  \leq   n+1.
\]
Consequently, $\mfv(G)  \leq  2n - 2\gamma^{ID}(G)  \leq 2n - 2\mfv(G)$
and thus, we get 
$ \mfv(G)  \leq  \frac{2}{3}n,$
which gives our claim.
 \end{proof}

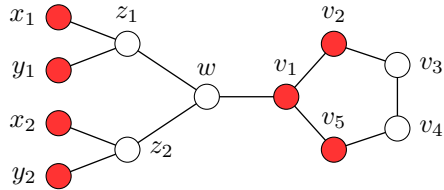
\begin{figure}[h]
     \centering
    \begin{tikzpicture}[yscale=0.7, xscale=0.7, rotate=0]

    \renewcommand*{\EdgeLineWidth}{ 0.5pt}
    \SetVertexMath
    \SetVertexLabelOut

    {
    \SetVertexNoLabel
    \tikzstyle{VertexStyle}=[minimum size=1pt,inner sep=0pt]

    \Vertex[x=3, y=2]{1}
    \Vertex[x=3.9, y=3]{2}
    \Vertex[x=5.1, y=2.6]{3}
    \Vertex[x=5.1, y=1.4]{4}
    \Vertex[x=3.9, y=1]{5}
    \Vertex[x=1.5, y=2]{6}
    \Vertex[x=0, y=3]{7}
    \Vertex[x=-1.3, y=3.5]{8}
    \Vertex[x=-1.3, y=2.5]{9}
    \Vertex[x=0, y=1]{10}
    \Vertex[x=-1.3, y=0.5]{11}
    \Vertex[x=-1.3, y=1.5]{12}

    }

    \tikzstyle{EdgeStyle}=[bend right=0]
    \Edge(1)(2)
    \Edge(2)(3)
    \Edge(3)(4)
    \Edge(4)(5)
    \Edge(5)(1)
    \Edge(1)(6)
    \Edge(6)(7)    
    \Edge(6)(10)
    \Edge(7)(8)
    \Edge(7)(9)
    \Edge(10)(11)
    \Edge(10)(12)

    \tikzstyle{VertexStyle}=[circle,draw, fill=red!80]
    \SO[unit=0, Lpos=90](1){v_1}
    \SO[unit=0, Lpos=90](2){v_2}
    \SO[unit=0, Lpos=90](5){v_5}
    \SO[unit=0, Lpos=180](8){x_1}  
    \SO[unit=0, Lpos=180](11){y_2}
    \SO[unit=0, Lpos=180](12){x_2}
    \SO[unit=0, Lpos=180](9){y_1}

    \tikzstyle{VertexStyle}=[circle,draw, fill=white]
    \SO[unit=0, Lpos=0](3){v_3}
    \SO[unit=0, Lpos=0](4){v_4}
    \SO[unit=0, Lpos=90](6){w} 
    \SO[unit=0, Lpos=90](7){z_1}
    \SO[unit=0, Lpos=00](10){z_2}

\end{tikzpicture}

     \caption{The graph $G_2$. Here the highlighted vertices form a minimum ID code and all of them are also proper-min-forced.}
    \label{fig:lapio}
\end{figure}

In the next example, we will construct an infinite family of graphs with $2n/3-1$ proper-min-forced vertices when the order $n\ge 9$ is divisible by three. This amount of proper-min-forced vertices is \emph{just one off} from the bound $2n/3$ in Theorem~\ref{theo:bound}.  

\begin{example}\label{shovelgraph}
    Let $k \geq 1$.
    Let us define the  graph $G_k = (V(G_k), E(G_k))$ with
    $V(G_k) = \{w, v_1, v_2, v_3,v_4, v_5\} \cup \bigcup_{i = 1}^{k} \{x_i, y_i, z_i\}$
    and
    \[E(G_k) = \{wv_1, v_1v_2, v_2v_3, v_3v_4, v_4v_5, v_5v_1\} \bigcup_{i = 1}^{k} \{wz_i, x_iz_i, y_iz_i\}.\]
    The graph $G_2$ is illustrated in Figure~\ref{fig:lapio}.
    We claim that the vertices $v_1$, $v_2$, $v_5$ and $x_i$ and $y_i$ for all $i \in \{1, \dots, k\}$ are proper-min-forced.
    First, we will show that $\gamma^{ID} (G_k) = 2k + 3$.

    Let $S$ be an ID code in $G_k$. We have immediately the following two properties for $S$. 
    
    {\bf Fact 1:} We have $\abs{S \cap \{x_i, y_i, z_i\}} \geq 2$ for all $i$, because we need to dominate and distinguish the leaves $x_i$ and $y_i$.

    {\bf Fact 2:} We have $\abs{S \cap \{v_1, v_2, v_3, v_4, v_5\}} \geq 3$ since with at most two codewords we can obtain at most $3$ distinct nonempty $I$-sets, which is not enough to distinguish the four vertices $v_2, v_3, v_4, v_5$.

    Due to these two facts, we get $\gamma^{ID}(G_k) \geq 2k + 3$. Next we give an ID code that attains this bound. The set $S = \{v_1, v_2, v_5\} \cup \bigcup_{i = 1}^{k} \{x_i, y_i\}$ is an ID code in $G_k$ with $2k + 3$ vertices as seen by the following different and nonempty $I$-sets in $G_k$:  $I(S; v_1) = \{v_1, v_2, v_5\}$, $I(S; v_2) = \{v_1, v_2\}$,
    $I(S; v_3) = \{v_2\}$,
    $I(S; v_4) = \{v_5\}$,
     $I(S; v_5) = \{v_1, v_5\}$,
     $I(S; w) = \{v_1\}$,
     $I(S; x_i) = \{x_i\}$,
     $I(S; y_i) = \{y_i\}$,
     $I(S; z_i) = \{x_i, y_i\}$.
      It follows that $\gamma^{ID}(G_k) = 2k + 3$.

    Now we argue about the min-forcedness of the vertices in $G$.
    Let $S^*$ be a minimum ID code in $G$.
    First, we notice that the vertex $w$ is not in any minimum ID code, since including it would result in a code with at least $2k + 4$ vertices because of Facts 1 and 2.
    Next, we study vertices $x_i$, $y_i$ and $z_i$, and we keep in mind that we can select only two of them for a minimum ID code. 
    Choosing $x_i$ and $z_i$ yields $I(z_i) = \{x_i, z_i\} = I(x_i)$, and similarly does the choice $y_i$ and $z_i$.
    Hence, we need to choose $x_i$ and $y_i$ to any $S^*$.

    Now that the vertex $w$ is not dominated by itself or any vertex of type $z_i$, we must have $v_1$ in every minimum ID code.
    Then, to distinguish $w$ and $v_1$, we must have $\abs{S^* \cap \{ v_2, v_5\}} \geq 1$, and without loss of generality, let us assume that $v_2$ is in the code.
    Then $\abs{S^* \cap \{v_5, v_3\}} \geq 1$ to distinguish $v_1$ and $v_2$. 
    If we choose $v_3 \in S^*$ and examine the resulting code  $S^*$ of size $2k+3$, we notice that $I(S^*;w) = \{v_1\} = I(S^*;v_5)$. 
    Therefore, we are forced to choose $v_5 \in S^*$ and it turns out that the arbitrary minimum ID code $S^*$ is exactly $S$ that we defined above.
    Thus, we have shown that all vertices in $S$ are min-forced.

    It remains to show that the vertices in $S$ are also proper-min-forced. To do so, we verify that none of the min-forced vertices in $S$ is always-forced by providing an identifying code that does not contain it.
    We make use of the fact that the vertices $x_i$ and $y_i$ are open twins (clearly interchangeable) and that the vertices $v_2$ and $v_5$ are in symmetrical positions. 
    Let $S' = \{w, v_3, v_4, v_5\} \cup \bigcup_{i = 1}^{k} \{y_i, z_i\}$.
    We verify that $S'$ is an ID code of size $2k+4$ in $G_{k}$ by listing all $I$-sets in $G_k$:
 $I(S'; v_1) = \{w, v_5\}$,
        $I(S'; v_2) = \{v_3\}$,
        $I(S'; v_3) = \{v_3, v_4\}$,
        $I(S'; v_4) = \{v_3, v_4, v_5\}$,
         $I(S'; v_5) = \{v_4, v_5\}$,
        $I(S'; w) = \{w\} \cup \bigcup_{i = 1}^{k}\{z_i\}$,
        $I(S'; x_i) = \{z_i\}$,
        $I(S'; y_i) = \{y_i, z_i\}$,
        $I(S'; z_i) = \{w, y_i, z_i\}$.

    The min-forced vertices $x_i$, $v_1$ and $v_2$ are not in this ID code, making them proper-min-forced. 
    By symmetry, the vertices $v_5$ and $y_i$ are proper-min-forced as well.
    Consequently, $G_k$, $k\ge 1$, is a graph with $n=3k + 6$ vertices and $2k + 3=2n/3-1$ proper-min-forced vertices, thus \emph{almost} attaining the upper bound $2n/3$ of Theorem~\ref{theo:bound}.
\end{example}

We remark that in Example~\ref{shovelgraph} we utilized $P_3$ as an induced subgraph of $G_k$ on vertices $x_i$, $y_i$ and $z_i$ ($i=1,2,\dots,k)$. The path $P_3$ itself has $2n/3$ min-forced vertices (compare to Theorem~\ref{theo:bound}), but those vertices (which are the endpoints of $P_3$) are always-forced -- not proper-min-forced!

\section{An optimal result for dense graphs with proper-min-forced vertices}\label{DenseSec}

In this section, we determine the maximum number of edges in a graph having \emph{proper-min-forced} vertices; for simplicity, we focus on the graphs of even order.

For this purpose, we begin with the following useful lemma.
\begin{lemma}\label{lemma:XY}
    Let $m$ be an integer such that $m \geq 2$ as well as let $X = \{x_1, \ldots, x_m\}$ and $Y = \{y_1, \ldots, y_m\}$ be subsets of $V(G)$ such that $N(x_i) = N(y_i) = V(G) \setminus \{x_i, y_i\}$ for each $i \in \{1, \ldots, m\}$.
    \begin{itemize}
        \item[(i)] If $S$ is an identifying code in $G$, then $|S \cap (X \cup Y)| \geq 2m-1$. 
        \item[(ii)] Moreover, if $S$ is an identifying code in $G$ and there exists a universal vertex $u \in V(G)$, then $X \cup Y \subseteq S$.
    \end{itemize}
\end{lemma}
\begin{proof}
       Suppose that there are two vertices $x_i$ and $x_j$ that are not in $S$. Since $N[y_i]\, \triangle \, N[y_j] = \{x_i, x_j\}$ and, therefore, $I(y_i) = S = I(y_j)$, this contradicts the fact that $S$ is an identifying code.
    Therefore, there cannot be two vertices in $X$ that are not in $S$.
    Similarly at least one of the vertices $y_i$ and $y_j$ must be in $S$ for all pairs $i \neq j$.
    Now suppose that there are two vertices $x_i, y_j \notin S$. Since $N[x_i]\, \triangle \, N[y_i] = \{x_i, y_i\}$,  it implies that $\abs{\{x_i, y_i\} \cap S}\geq 1$  for all $i \in \{1, \ldots, m\}$. Therefore, we can assume $i\neq j$. Subsequently, $I(x_j) = S = I(y_i)$, giving a contradiction again. This yields the claim (i).

    Let $u$ be a universal vertex in $G$. Clearly, $u\notin X\cup Y$. If $x_i\notin S$ (resp. $y_j\notin S$), then $I(y_i)=S=I(u)$ (resp.  $I(x_j)=I(u))$ which is not possible. This gives the assertion (ii).
     \end{proof}

In what follows, $G_1 \cup G_2$ denotes the disjoint union of the graphs $G_1$ and $G_2$.
For the rest of the section, we name the vertices of disjoint unions of the graph $K_2$ as $x_i$ and $y_i$ (similar to Lemma~\ref{lemma:XY}) when applicable.

\begin{theorem}
    Let $n = 2k$, where $k\ge 6$. There exists a graph of order $n$ containing a proper-min-forced vertex with $\binom{2k}{2}-(k+2)$ edges. 
\end{theorem}
\begin{proof}
Denote the bull graph by $\bull$ where $V(\bull)=\{v_1,v_2,v_3,v_4,v_5\}$ and 
\[E(\bull)=\{v_1v_2,v_1v_3,v_2v_3, v_2v_5,v_3v_4\}.\]

Let us consider a graph $G_{2k}$ defined via its  complement 
$\overline{G}_{2k} = \bull  \cup K_1\cup  \bigcup_{i=1}^{k-3} K_2$, 
where the last part means $k-3$ disjoint unions of $K_2$ and the vertices in the disjoint unions are named $X=\{x_i\mid i=1,2,\dots, k-3\}  $ and $Y=\{y_i \mid i=1,2,\dots, k-3\}$. That is, $x_i$ and $y_i$  are adjacent in the graph $\overline{G}_{2k}$ (illustrated in Figure~\ref{fig:bull-complement}).

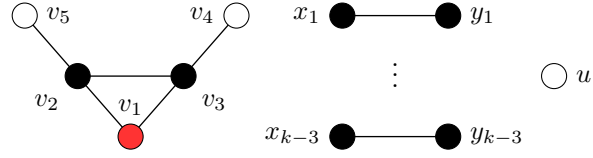
\begin{figure}
    \centering
    \begin{tikzpicture}[yscale=0.8, xscale=0.7, rotate=0]

    \renewcommand*{\EdgeLineWidth}{ 0.5pt}
    \SetVertexMath
    \SetVertexLabelOut

    {
    \SetVertexNoLabel
    \tikzstyle{VertexStyle}=[minimum size=1pt,inner sep=0pt]

    \Vertex[x=2, y=0]{1}
    \Vertex[x=1, y=1]{2}
    \Vertex[x=3, y=1]{3}
    \Vertex[x=4, y=2]{4}
    \Vertex[x=0, y=2]{5}
    \Vertex[x=6, y=2]{6}
    \Vertex[x=8, y=2]{7}
    \Vertex[x=6, y=0]{8}
    \Vertex[x=8, y=0]{9}
    \Vertex[x=10, y=1]{10}

    }

\begin{scope} 
        \tikzstyle{VertexStyle}=[rectangle, fill=white,minimum size=15pt,inner sep=3pt]
        \SetVertexLabelIn
        \Vertex[L=\vdots,x=7,y=1.15] {dots}
    \end{scope}
    \tikzstyle{EdgeStyle}=[bend right=0]
    \Edge(1)(2)
    \Edge(1)(3)
    \Edge(2)(3)
    \Edge(3)(4)
    \Edge(2)(5)
    \Edge(6)(7)  
    \Edge(8)(9)

    \tikzstyle{VertexStyle}=[circle,draw, fill=red!80]
    \SO[unit=0, Lpos=90](1){v_1}

    \tikzstyle{VertexStyle}=[circle,draw, fill=black]
    \SO[unit=0, Lpos=-45](3){v_3}
    \SO[unit=0, Lpos=180](6){x_1} 
    \SO[unit=0, Lpos=0](7){y_1}
    \SO[unit=0, Lpos=180](8){x_{k-3}}  
    \SO[unit=0, Lpos=0](9){y_{k-3}}
    \SO[unit=0, Lpos=225](2){v_2}

    \tikzstyle{VertexStyle}=[circle,draw, fill=white]
    \SO[unit=0, Lpos=180](4){v_4}
    \SO[unit=0, Lpos=0](5){v_5}
    \SO[unit=0, Lpos=00](10){u}

\end{tikzpicture}
    \caption{The complement graph $\overline{G}_{2k}$.}
    \label{fig:bull-complement}
\end{figure}

The vertices $x_i$ and $y_i$ are always-forced by Lemma~\ref{lemma:XY}(ii), since the vertex (named as $u$) corresponding to $K_1$ in $\overline{G}_{2k}$ is a universal vertex in $G_{2k}$.
The vertices $v_2$ and $v_3$ (see Figure~\ref{fig:bull-complement}) are always-forced because $N[v_1]\, \triangle \, N[v_4] = \{v_2\}$ and $N[v_1]\, \triangle \, N[v_5] = \{v_3\}$.

Let us denote $\AFV = X \cup Y \cup \{v_2, v_3\}$.
Hence, all the vertices in $\AFV$ are always-forced.
There are $2k-4$ vertices in $\AFV$, and clearly $\gamma^{ID}(G_{2k}) \geq \abs{\AFV} = 2k-4$.
The set $\AFV$ is not an ID code in $G_{2k}$, because $I(\AFV;v_2)=I(\AFV;v_4)$ and $I(\AFV;v_3)=I(\AFV;v_5)$; in fact, it is straightforward to check that the pairs $v_2, v_4$ and $v_3, v_5$ are the only pairs of vertices that are not distinguished by $A$.
Let us consider codes with $2k-3$ vertices.
The sets $\AFV \cup \{u\}$, $\AFV \cup \{v_5\}$ and $\AFV \cup \{v_4\}$ are not ID codes because at least one of the pairs $v_2, v_4$ and $v_3, v_5$ is not distinguished. However, the vertex $v_1$ distinguishes both pairs.
Consequently, the set $\AFV \cup \{v_1\}$ \emph{is} an ID code, and therefore, $\gamma^{ID}(G_{2k}) = 2k-3$.
Therefore, $v_1$ belongs to every minimum ID code, since there is the unique such code.

Finally, we show that the vertex $v_1$ is proper-min-forced by showing that $S=V(G_{2k}) \setminus \{v_1\}$ is indeed an ID code in $G_{2k}$, that is, the vertex $v_1$ is \emph{not} always-forced. 
The $I$-sets with respect to $S$ are nonempty and distinct: $I(S; v_1) = V(G_{2k}) \setminus \{v_1, v_2, v_3\}$, $I(S; v_2) = V(G_{2k}) \setminus \{v_1, v_3, v_5\}$, $I(S; v_3) = V(G_{2k}) \setminus \{v_1, v_2, v_4\}$, $I(S; v_4) = V(G_{2k}) \setminus \{v_1, v_3\}$, $I(S; v_5) = V(G_{2k}) \setminus \{v_1, v_2\}$, $I(S; x_i) = V(G_{2k}) \setminus \{v_1, y_i\}$, $I(S; y_i) = V(G_{2k}) \setminus \{v_1, x_i\}$ and $I(S; u) = V(G_{2k}) \setminus \{v_1\}$. Therefore, $v_1$ is proper-min-forced.
 \end{proof}

In the following theorem we show that the graphs in the previous proof are (some of) the densest graphs of even order that contain proper-min-forced vertices.
\begin{theorem}
    Let $n = 2k$ where $k\ge 6$. Graphs of order $n$ with at least $\binom{2k}{2}-(k+1)$ edges have no proper-min-forced vertices.
\end{theorem}
\begin{proof} Let us first make three useful observations for the complement of a closed-twin-free graph $G$:

\begin{observation}\label{obs:opentwins}
    There can be no open twins in $\overline{G}$ as they would be closed twins in $G$, making $G$ not closed-twin-free.
    In particular, there cannot be pendant vertices attached to the same parent vertex.
    As a special case, $\overline{G}$ cannot have the graph $P_3$ as a connected component.
\end{observation}

\begin{observation}\label{obs:criterion} Let $v$ be a vertex in $\overline{G}$ such that $d_{\overline{G}}(v)=\Delta(\overline{G}).$

    If \[\Delta(\overline{G}) \ge \abs{\{u \in V(\overline{G})\setminus\{v\} \sep d_{\overline{G}}(u) > 1\}}+2,\] then $G$ is not closed-twin-free. This is because a highest-degree vertex $v$ in $\overline{G}$ must have at least two degree neighbours with degree one, which makes them pendant vertices with the same parent $v$ (see Observation~\ref{obs:opentwins}).
\end{observation}

\begin{observation}\label{obs:univ-vertex}
    We notice that there cannot be more than one vertex with degree $0$ in the complement of $G$.
    Such a vertex is an universal vertex in $G$ and two or more of them would be closed twins.
\end{observation}
Observation~\ref{obs:univ-vertex} immediately rules out graphs with $\binom{n}{2}-(k-1)$ or more edges. Indeed, we have $\sum_{w\in V(\overline{G})}d_{\overline{G}}(w) = 2|E(\overline{G})|\le 2k-2$
by the handshaking lemma, and since the order of $\overline{G}$ equals $2k$, we have at least two vertices with degree $0$ in $\overline{G}$.

Let us look at graphs with $\binom{n}{2}-k$ edges.
In this case, possible degree sequences (that do not violate Observation~\ref{obs:univ-vertex}) for $\overline{G}$ in ascending order of degree are $(1,1,1,\dots,1,1,1)$ and $(0,1,1,\dots,1,1,2)$.
The latter degree sequence meets the criterion of Observation~\ref{obs:criterion}, which implies that $\overline{G}$ is not closed-twin-free.
The former sequence is realized only by the graph $G=K_{2k}$ with a perfect matching removed.
By Lemma~\ref{lemma:XY}, 
the graph has identifying number at least $2k-1$ and it is well-known (and easy to check) that its ID codes of the smallest size are $V(G) \setminus \{v\}$ for any $v \in V(G)$. Hence, there are no proper-min-forced vertices.

Now, let us consider graphs with $\binom{n}{2}-(k+1)$ edges.
Let us assume first that there are no universal vertices in $G$, in other words, that the smallest degree in $\overline{G}$ is $1$.
The possible degree sequences for $\overline{G}$ are $(1,\dots,1,1,2,2)$ and $(1,\dots,1,1,3)$, the latter of which meets the criterion of Observation~\ref{obs:criterion}, which implies that $\overline{G}$ is not closed-twin-free.
The degree sequence $(1,\dots,1,1,2,2)$ can only be realized in two nonisomorphic ways,  because the vertices of degree two can either be neighbours or not, namely, by the graphs  $P_4 \cup \bigcup_{i=1}^{k-2} K_2 $ and $P_3 \cup P_3 \cup \bigcup_{i=1}^{k-3} K_2 $.
The second one is not open-twin-free by Observation~\ref{obs:opentwins}.

Let us look at the first one $\overline{G} =  P_4 \cup \bigcup_{i=1}^{k-2} K_2 $, and let us name the vertices of $P_4$ naturally with $v_1$, $v_2$, $v_3$ and $v_4$. We begin by showing that $\gamma^{ID}(G)=2k-1$. 
Now, since $N_G[v_2]\, \triangle \, N_G[v_4] = \{v_1\}$ and $N_G[v_1]\, \triangle \, N_G[v_3] = \{v_4\}$, the vertices $v_1$ and $v_4$ are always-forced. 
By Lemma~\ref{lemma:XY} (and using its notation of the $k-2$ disjoint unions of $K_2$ in $\overline{G}$), at least $2(k-2)-1$ (all but one) of the vertices of type $x_i$ and $y_i$ are in an identifying code $S$.
Let us call the vertices $S_{XY} = S \cap (X \cup Y)$.
Now any identifying code $S$ in $G$ must have the set $S_{XY} \cup \{v_1, v_4\}$ as a subset.
The set $S_{XY} \cup \{v_1, v_4\}$ itself is not an identifying code due to the fact that 
 $I_G(v_1) = I_G(v_4) = S_{XY} \cup \{v_1, v_4\}$.
Since $N_G[v_1]\, \triangle \, N_G[v_4] = \{v_2, v_3\}$, at least one of the vertices $v_2$ and $v_3$ must be in $S$. Therefore, if $\abs{S_{XY}}=2k-4$, then $\abs{S} \geq 2k - 1$.
Hence, we may assume that $\abs{S_{XY}}=2k-5$, and without loss of generality, let $S_{XY} = (X \cup Y) \setminus \{x_1\}$.
If now $\abs{\{v_2, v_3\}\cap S}=1$, and say, without loss of generality, that $v_2 \in S$ and $v_3 \notin S$, then $I_G(S; v_4) = S_{XY} \cup \{v_1, v_2, v_4\} = S = I_G(S; y_1)$.
It follows that $\gamma^{ID}(G) \geq 2k-1$, and it is straightforward to verify that the sets $V(G) \setminus \{w\}$, where $w \notin \{v_1, v_4\}$, are ID codes for any choice of the vertex $w$.
This shows that there are no proper-min-forced vertices in $G$.

Now we allow one universal vertex in a graph with $\binom{n}{2}-(k+1)$ edges.
Possible degree sequences that have exactly one degree-0 vertex are $(0,1,\dots,1,2,2,2)$, $(0,1,\dots,1,1,2,3)$ and $(0,1,\dots,1,1,1,4)$.

Moreover, only the first sequence is possible since the two latter ones lead to graphs with closed-twins by Observation~\ref{obs:criterion}.
All possible ways to realize the valid degree sequence are the following:
\begin{itemize}
    \item $C_3 \cup K_1 \cup \bigcup_{i=1}^{k-2} K_2 $,
    \item $P_5 \cup K_1 \cup \bigcup_{i=1}^{k-3} K_2 $,
    \item $P_4 \cup P_3 \cup K_1 \cup \bigcup_{i=1}^{k-4} K_2 $,
    \item $P_3 \cup P_3 \cup P_3 \cup K_1 \cup \bigcup_{i=1}^{k-5} K_2. $  
\end{itemize}
The last two graphs have $P_3$ as a connected component and by Observation~\ref{obs:opentwins}, their complements are not closed-twin-free. We consider the other two graphs next.

Let us first look at an identifying code $S$ in the graph $G$ such that $\overline{G} = C_3 \cup K_1 \cup \bigcup_{i=1}^{k-2} K_2$.
We name the vertices in $C_3$ with $v_1$, $v_2$ and $v_3$.
We notice that $N_G[v_i]\, \triangle \, N_G[v_j] = \{v_i, v_j\}$ for $i, j \in \{1,2,3\}$.
To distinguish the vertices $v_i$ and $v_j$, we must have at least one of them in an identifying code. 
It follows that $\abs{\{v_1, v_2, v_3\} \cap S} \geq 2$.
Furthermore, by Lemma~\ref{lemma:XY}(ii) (with the usual notation for the disjoint unions of $K_2$), it must be that $X\cup Y \subseteq S$; in other words, all vertices of type $x_i$ and $y_i$ are always-forced.
Therefore, $\gamma^{ID}(G) \geq 2k -2$.
The codes $V(G) \setminus \{u, v_i\}$, where $u$ is the universal vertex and $v_i$ is one vertex in the cycle, are ID codes in $G$. By symmetry, it suffices to show this for the set $S=V(G) \setminus \{u, v_3\} = X \cup Y \cup \{v_1, v_2\}$ and, indeed, the $I$-sets are nonempty and distinct: $I_G(S; v_1) = S \setminus \{v_2\}$, $I_G(S; v_2) = S \setminus \{v_1\}$, $I_G(S; v_3) = S \setminus \{v_1, v_2\}$ and  $I_G(S; u) = S$ as well as $I_G(S; x_i) = S \setminus \{y_i\}$ and $I_G(S; y_i) = S \setminus \{x_i\}$ for all $i \in \{1, \dots, k-2\}$. In conclusion, the vertices $x_i$ and $y_i$ are always-forced, and the vertices $u$, $v_1$, $v_2$ and $v_3$ are not proper-min-forced.

Finally, let us look at an identifying code $S$ in the graph $G$ such that $\overline{G} = P_5 \cup K_1 \cup \bigcup_{i=1}^{k-3} K_2 $ and let us name the vertices of $P_5$ in order with $v_1, \dots, v_5$.
Again, the $2k - 6$ vertices in the set $X \cup Y$ are always-forced by Lemma~\ref{lemma:XY}(ii).
The vertices $v_2$ and $v_4$ are always-forced by the forcing pairs $v_3, v_5$ and $v_1, v_3$, respectively.
From the pairwise symmetric differences of $N_G[v_2]\, \triangle \, N_G[v_4]=\{v_1,v_5\}$, $N_G[u]\, \triangle \, N_G[v_2]=\{v_1,v_3\}$ and $N_G[u]\, \triangle \, N_G[v_4]=\{v_3,v_5\}$, we get that at least two of the vertices $v_1$, $v_3$ and $v_5$ must be in every identifying code of $G$.
It follows that $\gamma^{ID}(G) \geq 2k - 6 + 2 + 2 = 2k - 2$.
Let us now show that the sets $S_1 = V(G) \setminus \{v_1, u\}$, $S_3 = V(G) \setminus \{v_3, u\}$ and $S_5 = V(G) \setminus \{v_5, u\}$ are minimum identifying codes in $G$. Let us start with the $I$-sets of  $S_3$: $I(S_3; u) = S_3$, $I(S_3; v_1) = S_3 \setminus \{v_2\}$, $I(S_3; v_2) = S_3 \setminus \{v_1\}$, $I(S_3; v_3) = S_3 \setminus \{v_2, v_4\}$, $I(S_3; v_4) = S_3 \setminus \{v_5\}$ and $I(S_3; v_5) = S_3 \setminus \{v_4\}$ as well as  $I(S_3; x_i) = S_3 \setminus \{y_i\}$ and  $I(S_3; y_i) = S_3 \setminus \{x_i\}$ for all $i \in \{1, \dots, k-3\}$. The two remaining cases where either of the endpoints of $P_5$ is not in the code are symmetrical, so let us check the $I$-sets of $S_5$: $I(S_5; u) = S_5$, $I(S_5; v_1) = S_5 \setminus \{v_2\}$, $I(S_5; v_2) = S_5 \setminus \{v_1, v_3\}$, $I(S_5; v_3) = S_5 \setminus \{v_2, v_4\}$, $I(S_5; v_4) = S_5 \setminus \{v_3\}$ and $I(S_5; v_5) = S_5 \setminus \{v_4\}$ as well as $I(S_5; x_i) = S_5 \setminus \{y_i\}$ and $I(S_5; y_i) = S_5 \setminus \{x_i\}$ for all $i \in \{1, \dots, k-3\}$. Consequently, there are no proper-min-forced vertices in $G$. The assertion of the theorem now follows.
 \end{proof}

\section{On computational complexity}\label{CompSec}

We saw in Lemma~\ref{Connection} that determining whether a vertex is always-forced is (computationally) easy. However, this is not the case for proper-min-forced vertices. In what follows, we show that deciding whether a given vertex in a graph is proper-min-forced is co-NP-hard. The proof is based on a polynomial-time reduction from the well-known $3$-satisfiability ($3$-SAT) problem. For this purpose, we first introduce some notation related to the $3$-SAT problem.

Let $X = \{x_1, \dots, x_n\}$ be the set of variables and $\literals = \{x_1, \dots, x_n, \overline{x}_1, \dots, \overline{x}_n\}$ the set of literals, where $\overline{x}$ denotes the negation of $x$.
Let $\instance = \{ \clause_1, \clause_2, \cdots, \clause_m\}$ be an instance of the 3-SAT problem  over $X$, where each clause $\clause_j$ contains exactly three literals of $\literals$.
An assignment of truth values is a set $A \subseteq U$ where either $x_i \in A$ or $\overline{x}_i \in A$ for each $x_i \in X$; naturally, $x_i \in A$ is interpreted as \emph{true} and $\overline{x}_i \in A$ as \emph{false}.
The logical interpretation of these definitions is that $\instance$ corresponds to the formula $\clause_1 \wedge \clause_2 \wedge \cdots \wedge \clause_m$ and the clause $\clause_j = \{u_{j, 1},u_{j, 2},u_{j, 3}\} \subseteq \literals$ corresponds to $u_{j, 1} \vee u_{j, 2} \vee u_{j, 3}$. The truth assignment $\assignment$ satisfies the instance $\instance$ if each clause $\clause_j$ contains at least one element of $\assignment$, that is, $u_{j, 1} \vee u_{j, 2} \vee u_{j, 3}$ is true under the natural interpretation of the elements of $\assignment$.

\begin{theorem}
    The decision problem of determining whether a given vertex $w$ in a graph $G$ is proper-min-forced is co-NP-hard.
\end{theorem}
\begin{proof}

Let $\instance$ be an instance of the 3-SAT problem with $n$ variables and $m$ clauses.
Define a graph based on $\instance$ as follows.
For each variable $x_i$, we define a variable gadget $H_{x_i}$ with vertices $\{x_i, \overline{x}_i, a_i, b_i, c_i, d_i\}$ and edges $\{a_ib_i, b_ix_i, b_i\overline{x}_i, c_ix_i, c_i\overline{x}_i, c_id_i\}$.
For each clause $\clause_j$, we define a clause gadget $H_{\clause_j}$ with vertices $\{\alpha_j, \beta_j\}$ and a single edge $\alpha_j \beta_j$.
Then we combine these gadgets and two additional vertices $w$ and $v$ into a graph $G = (V(G), E(G))$ by defining (see Figure~\ref{fig:esim-G})
\[V(G) = \bigcup_{i=1}^{n} V(H_{x_i}) \cup \bigcup_{j=1}^{m} V(H_{\clause_j}) \cup \{w, v\}\]
and
\[
\begin{split}
    E(G) = &\bigcup_{i=1}^{n} E(H_{x_i}) \cup \bigcup_{j=1}^{m} E(H_{\clause_j})\cup \{\alpha_ju_{j,k} \sep k \in \{1,2,3\},\ 1\leq j \leq m\} \\ 
    &\cup \{w\alpha_j \sep 1\leq j \leq m\} \cup \{wv\} \text.
\end{split}
\]
The graph $G$ has $6n+2m+2$ vertices and $6n+5m+1$ edges.

\begin{figure}
    \centering
    \begin{tikzpicture}[yscale=1.2, xscale=1.4, rotate=0]

    \renewcommand*{\EdgeLineWidth}{ 0.5pt}
    \SetVertexMath
    \SetVertexLabelOut

    {
    \tikzstyle{VertexStyle}=[circle,draw, fill=white]

    \Vertex[x=-.3, y=2, L=v, Lpos=90]{v}
    \Vertex[x=1, y=2, L=w, Lpos=90]{w}
    \Vertex[x=2.33, y=3, L=\alpha_j, Lpos=90]{a1}
    {
    \tikzstyle{VertexStyle}=[minimum size=1pt,inner sep=0pt, fill=none]
    \Vertex[x=3.5, y=4, L=\ , Lpos=0]{x33}
    \Vertex[x=3.5, y=1.67, L=\ , Lpos=0]{x11}
    \Vertex[x=3.5, y=4, L=\ , Lpos=0]{x3}
    \Vertex[x=3.5, y=1.67, L=\ , Lpos=0]{x1}
    }
    {
    \tikzstyle{VertexStyle}=[diamond, scale=1.3, aspect=1.4, draw, fill=white]
    \EA[unit=0, L=\ ](x11){tim1}
    \EA[unit=0, L=\ ](x33){tim3}
    }
    
    \Vertex[x=4, y=3-0.33, L=x_i, Lpos=90]{x2}

    \Vertex[x=5.3, y=4-0.33, L=b_i, Lpos=90]{a}
    \Vertex[x=6.6, y=3-0.33, L=\overline{x}_i, Lpos=90]{x5}
    \Vertex[x=5.3, y=2-0.33, L=c_i, Lpos=90]{b}
\Vertex[x=5.3+1, y=2-0.5, L=d_i, Lpos=90]{c} 
      \Vertex[x=5.3+1, y=4-0.2, L=a_i, Lpos=90]{d}

    \Vertex[x=1., y=3.5, L=\beta_j, Lpos=90]{bee}

    }
    {
    \tikzstyle{VertexStyle}=[fill=white,circle]
    \SetVertexLabelIn

    \Vertex[x=2.33, y=1.5, L=\dots]{dots}

    }

    \tikzstyle{EdgeStyle}=[bend right=0]
    \Edge(v)(w)    
    \Edge(w)(a1)
    \Edge(w)(dots)    
    \Edge(a)(d) 
    \Edge(b)(c)

    \Edge(a1)(x2)
    \Edge(a1)(tim1)
    \Edge(a1)(tim3)
    
    \Edge(a1)(bee)
    
    \Edge(x2)(a)
    \Edge(x2)(b)
    \Edge(a)(x5)
    \Edge(b)(x5)

\end{tikzpicture}
    \caption{The vertex $w$ is connected to all vertices $\alpha_j$.
    The vertex $\alpha_j$ is connected to three variable gadgets, two of which are drawn as boxes for the sake of simplicity.}
    \label{fig:esim-G}
\end{figure}
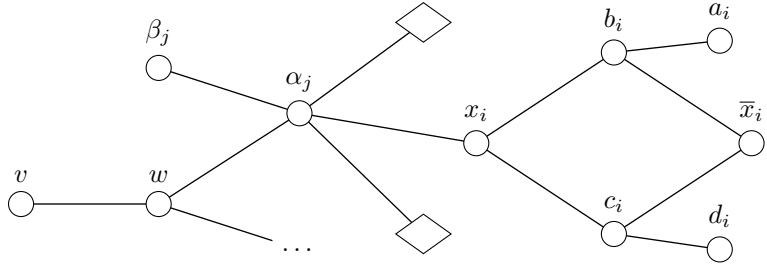

In what follows, we first show that $\gamma^{ID}(G) = 3n + m+ 1$. To prove the lower bound on $\gamma^{ID}(G)$, let $S$ be an ID code of $G$.
We can make the following observations for all $i \in \{1, 2, \ldots, n\}$ and $j \in \{1, 2, \ldots, m\}$:
\begin{enumerate}
    \item $a_i \in S$ or $b_i \in S$ (as $I(S, a_i) \neq \emptyset$),
    \item symmetrically, $c_i \in S$ or $d_i \in S$ (as $I(S, d_i) \neq \emptyset$),
    \item $x_i \in S$ or $\overline{x}_i \in S$ since $I(S;a_i) \neq I(S;b_i)$, and
    \item $\alpha_j \in S$ or $\beta_j \in S$ (as $I(S, \beta_j) \neq \emptyset$).
\end{enumerate}
By Observations 1--3, we get $\abs{S \cap V(H_{x_i})} \geq 3$ for all $i$ and by Observation 4, we get $\abs{S \cap V(H_{\clause_j})} \geq 1$ for all $j$.
Finally, $v \in S$ or $w \in S$ in order to dominate the vertex $v$; hence, $\abs{S\cap \{w,v\}}\geq 1$.
Thus, in conclusion, we get $\abs{S}\geq \gamma^{ID}(G) \geq 3n+m+1$.

Next, we will show that  $S' = \bigcup_{i=1}^{n} \{x_i, b_i, c_i\} \cup \bigcup_{j=1}^{m} \{\alpha_j\} \cup \{w\}$ is an ID code in $G$.
The $I$-sets are as follows: $I(S'; a_i) = \{b_i\}$,      $I(S'; b_i) = \{b_i, x_i\}$, $I(S'; c_i) = \{c_i, x_i\}$, $I(S'; d_i) = \{c_i\}$,
$I(S'; x_i) = \{b_i, c_i, x_i\} \cup \{\alpha_j \sep j: x_i \in \clause_j\}$,
$I(S'; \overline{x}_i) = \{b_i, c_i\} \cup \{\alpha_j \sep j: \overline{x}_i \in \clause_j\}$,
$I(S'; \alpha_j) = \{\alpha_j, w\} \cup ( \clause_j \cap X)$,
    $I(S'; \beta_j) = \{\alpha_j\}$,
    $I(S'; w) = \{w\} \cup \{\alpha_j \sep j\}$ and
     $I(S'; v) = \{w\}$.Subsequently, as $S'$ has $3n+m+1$ codewords, we obtain $\gamma^{ID}(G) = 3n + m+ 1$.

In what follows, we show that the vertex $w$ is proper-min-forced if and only if the 3-SAT instance $\instance$ is not satisfiable.
First, we observe that $w$ is not always-forced by showing that $V(G) \setminus \{w\}$ is an identifying code.
We can make use of the $I$-sets listed above, since $(S' \setminus \{w\}) \subseteq V(G) \setminus \{w\}$.
In particular, for $u \in T= \{a_i, b_i, c_i, d_i, x_i, \overline{x}_i, \beta_j\}$, we have $I(S'; u) \subseteq I(V(G)\setminus \{w\}; u)$. Hence, the $I$-sets $I(V(G)\setminus \{w\}; u)$ remain nonempty and distinct (among vertices in $T$) by the reasoning above.
The $I$-sets that need additional scrutiny are the following:
$I(V(G) \setminus \{w\}; \alpha_j) = \{\alpha_j, \beta_j\} \cup \clause_j$,
$I(V(G) \setminus \{w\}; w) = \{v\} \cup \{\alpha_j \sep j\}$ and $I(V(G) \setminus \{w\}; v) = \{v\}$.
It is easy to check that all $I$-sets are unique and nonempty. Therefore, $V(G) \setminus \{w\}$ is an ID code in $G$ and the vertex $w$ is not always-forced. Thus, what remains to be shown is that $w$ is min-forced if and only if the 3-SAT instance $\instance$ is not satisfiable.

($\Rightarrow$) Suppose to the contrary that $\instance$ is satisfiable and $\assignment$ is a valid truth assignment. In what follows, we show that $S_\assignment = \assignment \cup \bigcup_{i=1}^{n} \{ b_i, c_i\} \cup \bigcup_{j=1}^{m} \{\alpha_j\} \cup\{v\}$, with $3n+m+1 = \gamma^{ID}(G)$ vertices, is a minimum ID code in $G$.
Again, it is easy to verify that the $I$-sets with respect to $S_\assignment$ are nonempty and different:
$I(S_\assignment; a_i) = \{b_i\}$, $I(S_\assignment; b_i) = \{b_i\} \cup (\{x_i, \overline{x}_i\} \cap \assignment)$, $I(S_\assignment; c_i) = \{c_i\} \cup (\{x_i, \overline{x}_i\} \cap \assignment)$, $I(S_\assignment; d_i) = \{c_i\}$, $I(S_\assignment; x_i) = \{b_i, c_i\} \cup \{\alpha_j \sep j: x_i \in \clause_j\}\cup (\{x_i\} \cap \assignment)$,  $I(S_\assignment; \overline{x}_i) = \{b_i, c_i\}\cup \{\alpha_j \sep j: \overline{x}_i \in \clause_j\}\cup (\{\overline{x}_i\} \cap \assignment)$, $I(S_\assignment; \alpha_j) = \{\alpha_j\} \cup (\clause_j \cap \assignment)$, $I(S_\assignment; \beta_j) = \{\alpha_j\}$, $I(S_\assignment; w) = \{v\} \cup \{\alpha_j \sep j\}$ and $I(S_\assignment; v) = \{v\}$.
In particular, notice that $I(S_\assignment; x_i) \neq I(S_\assignment; \overline{x}_i)$ as $(\{x_i\} \cap \assignment) \neq (\{\overline{x}_i\} \cap \assignment)$ and $I(S_\assignment; \alpha_j) \neq I(S_\assignment; \beta_j)$ since $\clause_j \cap \assignment \neq \emptyset$ due to $A$ being a truth assignment that satisfies $\instance$.
Hence, the vertex $w$ is not min-forced and a contradiction follows.

($\Leftarrow$) Suppose to the contrary that $w$ in not min-forced, \textit{i.e.}, there exist a minimum ID code $S$ in $G$ such that $w \notin S$. Next, we ``extract'' a valid truth assignment from $S$. The set
$\bigcup_{i=1}^{n}\{x_i, \overline{x}_i\}\cap S$ contains exactly one of $x_i$ or $\overline{x}_i$ for all $i$ because of Observation 3 and $|S| = \gamma^{ID}(G) = 3n + m+ 1$.
This allows us to ``interpret'' it as a truth assignment of the $n$ variables like this: if $x_i \in S$, then we set the $i$th variable as true, and if  $\overline{x}_i \in S$, we set it as false.
Now we argue that this is a satisfying truth assignment.
Since $S$ is an identifying code, it must be that $I(S;\alpha_j) \neq I(S;\beta_j)$.
Since $N[\alpha_j] \,\triangle \, N[\beta_j] = \{u_{j,1}, u_{j,2}, u_{j,3}, w\}$ and we assumed that $w \notin S$, it must be that at least one of $u_{j,k} \in \clause_j$ is in $S$.
The vertex that distinguishes $\alpha_j$ from $\beta_j$ corresponds precisely to the true literal that satisfies clause $\clause_j$. 
This holds for all pairs $\alpha_j, \beta_j$, therefore there is a true literal in every clause $\clause_j$ and it follows that $\instance$ is satisfiable.
This concludes the other direction of the claim: if $w$ is not proper-min-forced, then $\instance$ is satisfiable. 
\end{proof}

\section{Conclusions}
In this paper, we studied the vertices that have to be in every minimum identifying code. As the always-forced vertices in a graph are rather easy to handle, we focused on the more complicated question of the proper-min-forced vertices. While the number of always-forced vertices can reach the trivial upper bound $\gamma^{ID}(G)\le n-1$, we showed that for  proper-min-forced vertices there is the absolute boundary $2n/3$ on their cardinality. We also gave an infinite family of graphs with $2n/3-1$ such vertices.  It is tempting to argue that studying the structure of  colour graphs more carefully  could show the upper bound $2n/3-1$, but we believe that it might not be enough and new ideas are needed if the optimal upper bound is actually $2n/3-1$. For future work, it would be interesting to know whether $2n/3-1$ or $2n/3$ is the optimal bound.

We also examined the maximum number of edges in a graph that contains proper-min-forced vertices. We gave an upper bound for graphs of even order  and showed that in can be attained.
Finally, we considered the algorithmic complexity of deciding whether a vertex is proper-min-forced or not. We showed that this problem is co-NP-hard. Recall that deciding whether a vertex is always-forced is algorithmically easy.

\medskip

    \textbf{Acknowledgments:}
 The authors are partially supported by the Research Council of Finland  grants 338797 and 358718.


\end{document}